\documentclass[authoryear,hyperref,noinfoline]{imsart}

\RequirePackage{graphicx}
\RequirePackage{color}

\RequirePackage[OT1]{fontenc}
\RequirePackage{amsthm,amsmath}
\RequirePackage{natbib}
\numberwithin{equation}{section}
\RequirePackage[colorlinks,citecolor=blue,urlcolor=blue]{hyperref}
\RequirePackage{hypernat}

\usepackage{amssymb}
\RequirePackage{url}
\RequirePackage{ifthen}
\usepackage{lineno}
\usepackage{subfigure}
\RequirePackage{multirow}
\RequirePackage{hangcaption}
\theoremstyle{plain}
\newtheorem{Thm}{Theorem}[section]   % STARTS each SECTION
\newtheorem{Prop}[Thm]{Proposition}
\newtheorem{Lem}[Thm]{Lemma}
\newtheorem{Cor}[Thm]{Corollary}

\theoremstyle{rem}
\newtheorem{Rem}[Thm]{Remark}
\newcommand{\rfia}[1]{\makebox[\parindent][l]{%
                     \makebox[0em][r]{\rm(}\sf#1\rm)}}
\newcounter{ABCcB}
\newcommand{\theABCcC}{\alph{ABCcB}}

\newenvironment{ABC}{\begin{list}{%    VE LABELWIDTH 1em + LABELSEP 0.5em
  \rfia{\theABCcC}}{\usecounter{ABCcB} \topsep 0ex \partopsep 0ex \itemsep0ex
  \parsep=\parskip \leftmargin 0em \rightmargin 0em \itemindent=\parindent
  \listparindent=\parindent  \labelsep 0.5em \labelwidth 0.5em }}{\end{list}}

\newlength{\TabHeight}

\newcommand{\hreft}[1]{{\href{#1}{\tt\url{#1}}}}
\newcommand{\Ew}{\mathop{\rm {{}E{}}}\nolimits} % mathop erkennt nur (
\newcommand{\Var}{\mathop{\rm Var}\nolimits}    % extra \,space
\newcommand{\Cov}{\mathop{\rm Cov}\nolimits}    % extra \,space
\newcommand{\ve}{\varepsilon}

\newcommand{\SSs}{\scriptscriptstyle}
\newcommand{\Ts}{\textstyle}
\newcommand{\Ds}{\displaystyle}
\newcommand{\ssr}{\rm\scriptscriptstyle}
\newcommand{\Lw}{{\cal L}}

\newcommand{\tr}{\mathop{\rm{} tr{}}}
\newcommand{\B}{\mathbb B}

\newcommand{\R}{\mathbb R}

\newcommand{\N}{\mathbb N}
\newcommand{\EM} {{\mathbb I}}

\newcommand{\iid}{\mathrel{\stackrel{\ssr i.i.d.}{\sim}}}

\newcommand{\Tfrac}[2]{\textstyle\frac{#1}{#2}}

\newcommand{\rk}{\mathop{\rm rk}\nolimits}
\newlength{\SyW}
\newcommand{\wto}{\mathrel{\mathsurround0em \mbox{$\longrightarrow$}%
  \llap{\settowidth{\SyW}{$\longrightarrow$}%             \protect in @{}
  \raisebox{-.15ex}{\makebox[\SyW]{\scriptsize\rm w}}}}}

\newcommand{\Jc}{\mathop{\bf\rm{{}I{}}}\nolimits}
\newlength{\ei}\ei=0.0138888889em
\begin{document}
\begin{frontmatter}
%-------------------------------------------------------------------------------
\title{Optimally (Distributional-)Robust Kalman Filtering}
\runtitle{Optimally Robust Kalman Filtering}
%-------------------------------------------------------------------------------
\author{Peter Ruckdeschel}
\runauthor{Peter Ruckdeschel}
\ead[label=e1]{Peter.Ruckdeschel@itwm.fraunhofer.de}
\date{\today}
\affiliation{Fraunhofer ITWM and TU Kaiserslautern, Germany}

\address{Peter Ruckdeschel\\[1ex]
Fraunhofer ITWM, Abt.\ Finanzmathematik,\\ Fraunhofer-Platz 1, 67663 Kaiserslautern, Germany\\
and\\ TU Kaiserslautern, Fachbereich Mathematik, AG Statistik, \\
P.O.Box 3049, 67653 Kaiserslautern, Germany\\[1ex]
\printead{e1}
}
%-------------------------------------------------------------------------------
\begin{abstract}
We present optimality results for robust Kalman filtering where robustness
is understood in a distributional sense, i.e.; we enlarge the distribution assumptions
made in the ideal model by suitable neighborhoods. This allows for outliers which in our context
may be system-endogenous or -exogenous, which induces the somewhat conflicting
goals of tracking and attenuation.

The corresponding minimax MSE-problems are solved for both types of outliers
separately, resulting in closed-form saddle-points which consist of an
optimally-robust procedure and a corresponding least favorable outlier
situation.
The results are valid in a surprisingly general setup of state space models,
which is not limited to a Euclidean or time-discrete framework.

The solution however involves computation of conditional means in the ideal
model, which may pose computational problems.
In the particular situation that the ideal conditional mean is linear in
the observation innovation, we come up with a straight-forward Huberization, the
{\tt rLS} filter, which is very easy to compute.
For this linearity we obtain an again surprising characterization.

%Combining both optimal filter types (for system-endogenous and -exogenous situation)
%we come up with a delayed hybrid filter which is able to treat both types
%of outliers simultaneously.
\end{abstract}


\end{frontmatter}
%\linenumbers
%
%\maketitle
%
%-------------------------------------------------------------------------------
\section{Introduction}
%-------------------------------------------------------------------------------
Robustness issues in Kalman filtering have long been a research topic, with first
(non-verified) hits on a quick search for ``robust Kalman filter'' on
\hreft{scholar.google.com} as early as 1962 and 1967, i.e.; the former even before the
seminal \citet{Hu:64} paper, often referred to as birthday of Robust Statistics.

In the meantime there is an ever growing amount of literature on this topic
---\citet{Ka:Po:85} have already compiled as many as 209 references to that subject in 1985.
Excellent surveys are given in, e.g. \citet{Ka:Po:85}, \citet{St:Du:87},
\citet{Sch:Mi:94}, \citet{Kue:01}. %, and we are not attempting
%to give another exhaustive one.

In these references you find many different notions of robustness, all somewhat related
to stability but measuring this stability w.r.t.\ deviations of various ``input parameters'';
in this paper we are concerned with (distributional) \textit{minimax robustness}; i.e.;
we work with suitable distributional neighborhoods about an ideal model,
already used by \citet{B:S:93} and \citet{B:P:94}, and then
solve the problem to find the procedure minimizing the maximal predictive
inaccuracy on these neighborhoods---measured in terms of mean squared error
(MSE)---in quite generality, compare Theorems~\ref{ThmSO}, \ref{ThmeSO}, \ref{ThmISO}.
In the particular situation that the ideal conditional mean is linear in
the observation innovation (for a definition see subsection~\ref{KalmDef}),
the minimax filter is a straight-forward Huberization, the
{\tt rLS} filter, which is extremely easy to compute.
For this linearity we obtain a surprising characterization in
Propositions~\ref{nonorm} and \ref{nolin}. This motivates a corresponding
optimal test for linearity, Proposition~\ref{testprop}. Even in situations
where no or only partial knowledge of the size of the contamination is
available we can distinguish an optimal procedure, compare Lemma~\ref{Lem223}.

%-------------------------------------------------------------------------------
\section{General setup}\label{GenSetup}
%-------------------------------------------------------------------------------
\subsection{Ideal model} \label{idmodel}
%-------------------------------------------------------------------------------

In this section, we start with some definitions and assumptions.
We are working in the context of state space models (SSM's) as to be found
in many textbooks, cf.\ e.g.\ \citet{A:M:79}, \citet{Ha:91}, and \citet{D:K:01}.
%
%\subsubsection{Time Discrete, linear Euclidean Setup: }
\subsubsection{Time Discrete, linear Euclidean Setup}
The most prominent setting in this context is the linear,
time--discrete, Euclidean setup, which will serve as reference setting
in this paper: An unobservable $p$-dimensional
state $X_t$ evolves according to a possibly time-inhomo\-geneous
vector autoregressive model of order $1$  (VAR(1)) with
innovations $v_t$ and transition matrices $F_t$, i.e.,
\begin{equation} \label{VAR1}
X_t=F_t  X_{t-1}+ v_t
\end{equation}
The statistician observes a $q$-dimensional linear transformation
$Y_t$ of $X_t$ and in this makes an additive observation error $\ve_t$,
\begin{equation}
Y_t =Z_t  X_t+ \ve_t \label{linobs}
\end{equation}
In the ideal model we work in a Gaussian context, that is we assume
\begin{align}
&  v_t \sim {\cal N}_p(0,Q_t), \qquad %\label{normInnov}\\
\ve_t \sim {\cal N}_q(0,V_t),  \qquad  %\label{normError} \\
 X_0  \sim  {\cal N}_p(a_0,Q_0),  \label{normStart}\\
&{ X_0, v_s, \ve_t,  \;s,t\in\N\;\;\mbox{ stochastically independent}}
\end{align}
As usual,  normality assumptions  may be relaxed to working only
with specified first and second moments, if we restrict ourselves to
linear unbiased procedures as in the Gauss-Markov setting.\\
For this paper, we assume the hyper--parameters $F_t,Z_t,Q_t,V_t,a_0$ to be known.
%
%\subsubsection{Time Discrete, Hidden Markov Models: }

\subsubsection{Generalizations covered by the present approach}
Parts of our results (more specifically, all of sections~\ref{1stpOpt}, \ref{wayoutsec}) also cover much more general SSMs;
in this paragraph we sketch some of these. To
begin with, as long as MSE makes sense for the range of the states,
these results cover general Hidden Markov Models
for arbitrary observation space
 as given by %
\begin{align}
 &P(X_0\in A) =\int_A p_0^{X_0}(x)\, \mu_0(dx) \label{genInit}\\
 &P(X_t \in A| X_{t-1}=x_{t-1}, \ldots, X_0=x_0)=
 \int_A p_t^{X_t|X_{t-1}=x_{t-1}}(x) \,\mu_t(dx), \label{MarkovState}\\
 &P(Y_t \in B| X_t=x_t)= \int_B q_t^{Y_t|X_t=x_t}(y)\, \nu_t(dy) \label{GenObs}
\end{align}
In this setting, we assume known (and existing) [regular conditional] densities
$p_0^{X_0}$, $p_t^{\,\cdot\,|\,\cdot\,}$, $q_t^{\,\cdot\,|\,\cdot\,}$ w.r.t.\ known
measures $\nu_t$, $\mu_t$ on $\B^q$ and $\B^p$, respectively.
Dynamic (generalized) linear
models as discussed in \citet{W:H:M:85} and \citet{We:Ha:89} are covered as well
---under corresponding assumptions as to (conditional) densities and range of the states.
In applications of Mathematical Finance we also need to cover continuous time
settings, i.e.; there is an unobservable state evolving according to
an SDE
\begin{equation} \label{SDE}
 dX_t=  f(t,X_t)\,dt + q(t,X_t)\,dW_t
 \end{equation}
 where for $X_0$ we assume \eqref{genInit}, while  $W_t$, is a Wiener process,
and $f$ and $q$ are suitably measurable, known functions,
and observations $Y_t$ are either formulated as a time-continuous observation process (as in \citet{Ta:98})
or---more often---at discrete, but not necessarily equally spaced times,
compare, e.g.\ \citet{N:M:M:00} and \citet{Si:02}.
In this context, but also for corresponding non-linear time-discrete
SSMs, a straightforward approach linearizes the corresponding transition
and observation functions to give the \emph{(continuous-discrete) Extended Kalman Filter (EKF)}
 After this linearization
we are again in the context of a (time-inhomogeneous) linear SSM, hence
the methodology we develop in the sequel applies to this setting as well.

So far we do not cover approaches to improve on this simple linearization,
notably the \emph{second order nonlinear filter (SNF)} introduced
in \citet{Ja:70}, also cf.\ \citet[sec.~4.3.1]{Si:02}.
 the \emph{unscented Kalman filter (UKF)} \citep{J:U:DW:00}
and Hermite expansions as in \citet{Ai:02}, see also \citet[sec.~4.3]{Si:02}.
%For techniques to deal with non-linear time-discrete situations, see \citet{Ta:96}.

Going one more step ahead, to cover applications such as
portfolio optimization, we may allow for controls $U_t$   to be set or determined by
the statistician, and which are fed back in the state equations.
In the context of the continuous time model,
this is also known as SDEX, cf.\ \citet{N:M:M:00}, and
for the application of stochastic control to portfolio optimization, cf.\ \citet{Ko:97}.
In this setting, controls $U_t$ are usually assumed measurable w.r.t.\  $\sigma(Y_{t-})$;
to integrate them into our setting, we simply have to integrate them in the
corresponding condition vectors.

Finally, the question of specifying the order of conditioning left aside, we
do not make use of the linearity of time, so our minimax results also cover
suitable formulations of indirectly observed random fields.
%-------------------------------------------------------------------------------
\subsection{Deviations from the ideal model} \label{devidmod}
%-------------------------------------------------------------------------------
As usual with Robust Statistics, the ideal model assumptions we have specified
so far are extended by allowing (small) deviations, most prominently generated by outliers.
In our notation, suffix ``${\Ts \rm id}$'' indicates the {\it id}eal setting, ``${\Ts \rm di}$''
the {\it di}storting (contaminating) situation, ``${\Ts \rm re}$'' the {\it re}alistic, contaminated situation.

\subsubsection{AO's and IO's}
In SSM context (and contrary to the independent setting), outliers may or may not propagate.
Following the terminology of \citet{F:72}, we distinguish {\em innovation outliers}
(or IO's) and {\em additive outliers} (or AO's).  Historically, AO's denote gross errors affecting the observation
errors, i.e.,
\begin{eqnarray}
{\rm AO} &::& \ve^{\SSs\rm re}_t \sim
(1-r_{\SSs\rm AO}){\cal L}(\ve^{\SSs\rm id}_t)+
 r_{\SSs\rm AO} {\cal L}(\ve^{\SSs\rm di}_t)   \label{AOdef}
\end{eqnarray}
where ${\cal L}(\ve^{\SSs\rm di}_t)$ is arbitrary, unknown and uncontrollable (a.u.u.) and $0\leq r_{\SSs\rm AO}\leq 1$
is the AO-contamination radius, i.e.; the probability for an AO.
IO's on the other hand are usually defined as outliers which affect the innovations,
\begin{eqnarray}
{\rm IO} &::& v^{\SSs\rm re}_t  \sim
(1-r_{\SSs\rm IO}){\cal L}(v^{\SSs\rm id}_t)+
r_{\SSs\rm IO} {\cal L}(v^{\SSs\rm di}_t) \label{IOdef}
\end{eqnarray}
where again ${\cal L}(v^{\SSs\rm di}_t)$ is a.u.u. and $0\leq r_{\SSs\rm IO}\leq 1$
is the corresponding radius.

We stick to this distinction for consistency with literature, although we rather
use these terms in a wider sense, unless explicitly otherwise stated:
\textit{IO}'s denote endogenous outliers affecting the state equation
in general, hence distortion propagates into subsequent states. This also covers level shifts or linear trends;
which if $|F_t|<1$ are not included in \eqref{IOdef}, as IO's would then decay geometrically in $t$.
We also extend the meaning of \textit{AO}'s to denote general exogenous outliers which
 enter the observation equation only and thus do not propagate, like substitutive outliers or \textit{SO}'s defined as
\begin{eqnarray}
{\rm SO} &::& Y^{\SSs\rm re}_t  \sim (1-r_{\SSs\rm SO}){\cal
L}(Y^{\SSs\rm id}_t)+r_{\SSs\rm SO}{\cal
L}(Y^{\SSs\rm di}_t) \label{SOdef}
\end{eqnarray}
where again ${\cal L}(Y^{\SSs\rm di}_t)$ is a.u.u.\   and $0\leq r_{\SSs\rm SO}\leq 1$
is the corresponding  radius.

Apparently, the SO-ball of radius $r$ consisting of all ${\cal L}(Y^{\SSs\rm re}_t)$
according to \eqref{SOdef} contains the corresponding AO-ball of the same radius
when $Y^{\SSs\rm re}_t=Z_tX_t+ \ve^{\SSs\rm re}_t$. However, for technical reasons,
we make the additional assumption that
\begin{equation}\label{indep}
Y^{\SSs\rm id}_t, Y^{\SSs\rm di}_t\quad \mbox{stochastically independent}
\end{equation}
and then this relation no longer holds.

\subsubsection{Different and competing goals induced by endogenous and exogenous outliers}

In the presence of AO's we would like to attenuate their effect, while
when there are IO's, the usual goal in online applications would be tracking, i.e.;  detect
 structural changes as fast as possible and/or react on the changed situation.
 A situation where both AO's and IO's may occur poses an
 identification problem:
 Immediately after a suspicious observation we cannot tell IO type from AO type.
 Hence a simultaneous treatment of both types will only be
 possible with a certain delay---see \citet{Ru:10}.
%

%
%-------------------------------------------------------------------------------
\subsection{Classical Method: Kalman--Filter} \label{classKalmss}
%-------------------------------------------------------------------------------

\subsubsection{Filter Problem}
The most important problem in SSM formulation is to reconstruct the unobservable states $X_t$ based
on the observations $Y_t$. For abbreviation let us denote
\begin{equation}
Y_{1:t}=(Y_1,\ldots,Y_t), \quad Y_{1:0}:=\emptyset
\end{equation}
Then using MSE risk, the optimal reconstruction is distinguished as
\begin{equation}
\Ew \big| X_t-f_t\big|^2 = \min\nolimits_{f_t},\qquad f_t \mbox{ measurable w.r.t.\ } \sigma(Y_{1:s})  \label{MSEpb}
\end{equation}
Depending on $s$ this is a prediction ($s<t$),  a filtering ($s=t$)
and a smoothing problem ($s>t$). In the sequel we
will confine ourselves to the filtering problem.
{\subsubsection{Kalman--Filter} \label{KalmDef}
It is well-known that the general solution to \eqref{MSEpb} is the corresponding
conditional expectation  $\Ew[ X_t|Y_{1:s}]$. Except for the Gaussian case, this exact
conditional expectation may be computational too expensive.
Hence similar to  the Gauss-Markov setting, it is common to restrict
oneself to linear filters. In this context, the seminal work of \citet{Kal:60} (discrete-time setting)
and \citet{K:B:61} (continuous-time setting) introduced effective schemes to
compute this optimal linear filter $X_{t|t}$. In discrete time, we reproduce it here for
later reference:
\begin{align}%[lrclrcl]
\!\!\!\!\!\!\mbox{Init.: }&\!\!\!\!&
 X_{0|0} &= a_0,\qquad &\Sigma_{0|0}&=Q_0 \label{bet1}
 \\
\!\!\!\!\!\!\mbox{Pred.: }&\!\!\!\!
& X_{t|t-1}&= F_t  X_{t-1|t-1} , \qquad &\Sigma_{t|t-1}&=F_t\Sigma_{t-1|t-1}F_t^\tau + Q_t\label{bet2}%\\
%\!\!\!\quad& \color{gray}[\Delta x_t&\color{gray}= x_t- x_{t|t-1}]%\nonumber
\\
\!\!\!\!\!\!\mbox{Corr.: }%\nonumber\\
&\!\!\!\!& X_{t|t}&=  X_{t|t-1} + M^0_t \Delta Y_t, \label{bet3}%\\
\!\!\!\quad&  \Sigma_{t|t} &= (\EM_p-M^0_tZ_t)\Sigma_{t|t-1}\\
\!\!\!\!\!\!\mbox{for}&&\Delta X_t &=X_{t}-X_{t|t-1},
\!\!\!\quad&  \Delta Y_t &= Y_t-Z_t X_{t|t-1}=Z_t \Delta X_{t}+\ve_t,\nonumber\\
&&\Delta_t&=Z_t \Sigma_{t|t-1}Z_t^\tau + V_t,
\!\!\!\quad& M^0_t&=\Sigma_{t|t-1}Z_t^\tau \Delta_t^{-} 
\end{align}
and where $\Delta X_t$ is the prediction error, 
$\Delta Y_t$ the observation innovation, and
$\Sigma_{t|t}=\Cov(\Delta X_t)$, $\Sigma_{t|t-1}=\Cov(X_t-X_{t|t-1})$,
$\Delta_t = \Cov(\Delta Y_t)$;
$M^0_t$ is the so-called \textit{Kalman gain}, and
$\Delta_t^-$ stands for the Moore-Penrose inverse of $\Delta_t$.
%
%-------------------------------------------------------------------------------
\subsubsection{Optimality of the Kalman--Filter}\label{optKalm}
%-------------------------------------------------------------------------------

Realizing that $M^0_t \Delta Y_t$ is an orthogonal projection, it is not hard to see %(compare, e.g. \citet{Ru:00,Ru:01})
that the (classical) Kalman filter solves problem~\eqref{MSEpb} (for $s=t$)
among all linear filters.
Using orthogonality of $\{\Delta Y_t\}_t$ once again, we may setup similar recursions for
the corresponding best linear smoother; see, e.g.\ \citet{A:M:79}, \citet{D:K:01}.
Under normality, i.e.; assuming \eqref{normStart},
we even have $ X_{t|t[-1]}=\Ew[ X_t|Y_{1:t[-1]}]$, i.e.; the Kalman filter is optimal among all
$Y_{1:t[-1]}$-measurable filters. It also is the posterior mode of ${\cal L}( X_t|Y_{1:t})$
and  $ X_{t|t}$ can also be seen to be the ML estimator for a regression model
with random parameter; for the last property, compare \citet{D:H:72}.
%-------------------------------------------------------------------------------
\subsubsection{Features of the Kalman--Filter}
%-------------------------------------------------------------------------------

The Kalman filter stands out for its clear and understandable structure:
it comes in three steps, all of which are linear, hence cheap to evaluate and
easy to interpret. Due to the Markovian structure of the state equation,
all information from the past useful for the future may be captured in the
value of $X_{t|t-1}$, so only very limited memory is needed.\\
From a (distributional) Robustness point of view, this linearity at the same time is a weakness
of this filter---$y$ enters unbounded into the correction step which
hence is prone to outliers.
A good robustification of this approach would try to retain as much as
possible from these positive properties of the Kalman filter while revising
the unboundedness in the correction step.
%-------------------------------------------------------------------------------
\section{The rLS as optimally-robust filter} \label{rLSsec}
%-------------------------------------------------------------------------------
\subsection{Definition}\label{rLSDefss}
%-------------------------------------------------------------------------------
\subsubsection{{\it r}obustifying {\it r}ecursive  {\it L}east {\it S}quares: rLS}
In a first step we limit ourselves to AO's. Notationally, where clear from
the context, we suppress the time index $t$.
As no (new) observations
enter the initialization and prediction steps, these steps may be left unchanged.
In the correction step, we will have to modify the orthogonal projection
present in \eqref{bet3}.
Suggested by H.~Rieder and worked out in \citet[ch.~2]{Ru:01},
the following robustification of the correction step
is straightforward:
Instead of $M^0\Delta Y$, we use a Huberization of this correction % \vspace{1ex}
\begin{equation} \label{HbDef}
H_b(M^0\Delta Y) =  M^0\Delta Y \min\{1, %\frac{b}{\big|M^0\Delta y\big|}
                b/\big|M^0\Delta Y\big|\}
\end{equation}
for some suitably chosen clipping height $b$. Apparently, this proposal removes
the unboundedness problem of the classical Kalman filter while still remaining
reasonably simple, in particular this modification is non-iterative, hence especially useful
for online-purposes.\\
% However it should be noted that, departing from the Kalman filter and
%at the same time insisting on strict recursivity,
%we possibly exclude ``better'' non-re\-cur\-sive procedures, compare Remark~\ref{RemNonRec}. These procedures on the other
%hand would be much more expensive to compute.

%\begin{Rem}\rm\small  %  $\big|\,\cdot\,\big|$ in expression $\big|M^0\Delta Y\big|$ denotes the Euclidean norm of $\R^q$; instead,
%  however you could also use other norms like a Mahalanobis type norm. With respect
%  to Theorem~\ref{ThmSO}, optimality is preserved when instead of the Euclidean norm
%  used in the MSE, you use the corresponding alternative norm.
%\end{Rem}
\subsubsection{Choice of the clipping height $b$}
For the choice of the clipping height $b$, we have two proposals. Both are based on the simplifying assumption that
$\Ew_{\rm\SSs id}  [\Delta X | \Delta Y]$ is linear,
which will turn out to  only be approximately right.
The first one, an Anscombe criterion, chooses $b=b(\delta)$ such that
    \begin{equation}
     \Ew_{\rm\SSs id} \big|\Delta  X-H_b(M^0 \Delta Y)\big|^2
         \stackrel{!}{=}(1+\delta) \Ew_{\rm\SSs id} \big|\Delta  X-M^0 \Delta Y\big|^2
         \label{deltakrit}
    \end{equation}
 $\delta$ may be interpreted as ``insurance premium''
to be paid in terms of loss of efficiency in the ideal model compared to the optimal procedure in this (ideal) setting,
i.e.; the classical Kalman filter.

The second criterion for a given radius $r\in[0,1]$  of the
(SO-) neighborhood ${\cal U}^{\rm\SSs SO}(r)$  determines $b=b(r)$ such that
    \begin{equation}
     (1-r) \Ew_{\rm\SSs id} (|M^0 \Delta Y |-b)_+ \stackrel{!}{=} r b
         \label{deltakrit2}
    \end{equation}
Assuming linear ideal conditional expectations, this will produce the minimax-MSE procedure for ${\cal U}^{\rm\SSs SO}(r)$ according to Theorem~\ref{ThmSO} below.

One might object that \eqref{deltakrit2} assumes $r$ to be known, which in practice hardly ever is true. If $r$ is unknown however,
we translate an idea worked out in \citet{RKR08}: Assume we have limited knowledge about $r$, say $r\in[r_l,r_u]$, $0\leq r_l<r_u\leq 1$.
Then we distinguish a \textit{least favorable radius} $r_0$ defined in the following expressions
\begin{eqnarray}
r_0&=&{\rm argmin}_{s \in [r_l,r_u]}\rho_0(s),\qquad \rho_0(s) = \max_{r\in[r_l,r_u]} \rho(r,s), \label{r01} \label{rsline}\\
\rho(r,s)&=&\frac{\max_{P\in{\cal U}^{\rm\SSs SO}(r)}{\rm MSE}_P\big({\rm rLS}(b(s))\big)}{\max_{P'\in {\cal U}^{\rm\SSs SO}(r)}{\rm MSE}_{P'}\big({\rm rLS}(b(r))\big)}\label{r02}
\end{eqnarray}
and use the corresponding $b(r_0)$. Procedure ${\rm rLS}(b(r_0))$ then minimizes the maximal inefficiency $\rho_0 (s)$ among all procedures ${\rm rLS}(b(r))$, i.e.;
each rLS for some clipping height $b(r)\not=b(r_0)$ has an inefficiency no smaller than $\rho_0(r_0)$ for some $r' \in[r_l,r_u]$.
%
%However, contrary to the asymptotic setup of the infinitesimal
%Robustness as used in the cited reference, in our context, the minimal clipping height $b_{\rm \SSs min}=b(r=\infty)$ possible for the rLS
%filter is $0$,
%so if we let $r,s$ vary in $[0,\infty)$, we would always use radius $r_0=\infty$.
%
%So a more realistic approach would assume partial
%
Radius $r_0$ can be computed quite effectively by a bisection method: Let
\begin{eqnarray}
A_r&=&\Ew_{\rm\SSs id} \Big[ \tr \Cov_{\rm\SSs id}[\Delta X|\Delta  Y^{\rm\SSs id}] + (|M^0 \Delta Y^{\rm\SSs id}|-b(r))_+^2\Big] \label{Ardef}\\
B_r&=&\Ew_{\rm\SSs id}\Big[|M^0 \Delta Y^{\rm\SSs id}|^2 - (|M^0 \Delta Y^{\rm\SSs id}|-b(r))_+^2\Big] + b(r)^2 \label{Brdef}
\end{eqnarray}
Then the following analogue to \citet[Lemma~2.2.3]{Ko05} holds:
\begin{Lem}\label{Lem223}
In equations~\eqref{r01} and \eqref{r02}, let $r,s$ vary in $[r_l,r_u]$ with $0\leq r_l<r_u\leq 1$.
Then
  \begin{equation} \label{Lem223c}
  \rho_0(r)= \max\{A_r/A_{r_l},B_r/B_{r_u}\}
  \end{equation}
  and there exists some $\tilde r_0\in[r_l,r_u]$ such that
%  \begin{equation} \label{Lem223b}
$ A_{\tilde r_0}/A_{r_l}=B_{\tilde r_0}/B_{r_u}$.
%  \end{equation}
  This $\tilde r_0$ is least favorable, i.e.,
%  \begin{equation} \label{Lem223d}
$  \min_{r\in[r_l,r_u]} \rho_0(r)=\rho_0(\tilde r_0)$.
%  \end{equation}
Moreover, if $r_u=1$, $r_0=r_u$.
  \end{Lem}
  In particular, the last equality shows that one should restrict $r_u$ to be strictly smaller than $1$
  to get a sensible procedure.
%-------------------------------------------------------------------------------
\subsection{(One-Step)-Optimality of the rLS}\label{1stpOpt}
%-------------------------------------------------------------------------------
The (so-far) ad-hoc robustification proposed in the rLS filter has some remarkable optimality
properties: Let us first forget about the time structure and instead consider
the following  simplified, but general ``Bayesian'' model:\medskip\\ We have an unobservable
but interesting signal $X\sim P^X(dx)$, where for technical reasons we assume that in the ideal
model $\Ew |X|^2 <\infty$. Instead of $X$ we rather observe a random variable $Y$ taking values
in an arbitrary space of which
we know the ideal transition probabilities; more specifically, we assume that these ideal
transition probabilities for almost all $x$ have densities w.r.t.\ some measure $\mu$,
\begin{equation} \label{decM}
P^{Y|X=x}(dy)=\pi(y,x)\,\mu(dy)
\end{equation}
Our approach uses MSE as accuracy criterion for the reconstruction,
so is limited to ranges of $X$ where this makes sense. On the other hand it is this
reduction to the ``Bayesian'' model which makes the generalizations sketched in
section~\ref{idmodel} possible.
%,--- which essentially amounts to saying that the range of $X$ be a subset of some Hilbert space.
%
As (wide-sense) AO model,  we consider an SO outlier model, i.e.;
\begin{equation}
 Y^{\rm\SSs re} = (1-U)  Y^{\rm\SSs id} + U Y^{\rm\SSs di}, \qquad U\sim {\rm Bin}(1,r) \label{YSO}
\end{equation}
for $U$ independent of $(X,Y^{\rm\SSs id})$ and $(X,Y^{\rm\SSs di})$ and some distorting random variable $Y^{\rm\SSs di}$ for which, in a slight variation of condition~\eqref{indep}
 we assume
 \begin{equation}\label{indep2}
 Y^{\rm\SSs di},\; X\quad \mbox{stochastically independent}
 \end{equation}
and the law of which is arbitrary, unknown
and uncontrollable. As a first step consider the set $\partial{\cal U}^{\rm\SSs SO}(r)$ defined
as
\begin{equation}
\partial{\cal U}^{\rm\SSs SO}(r)=\Big\{{\cal L}(X,Y^{\rm\SSs re}) \,|\, Y^{\rm\SSs re} \;\mbox{acc. to \eqref{YSO} and \eqref{indep2}}\Big\}
\end{equation}
Because of condition~\eqref{indep2}, in the sequel we refer to the random variables $Y^{\rm\SSs re}$ and $Y^{\rm\SSs di}$ instead of their
respective (marginal) distributions only, while in the common gross error model as present in \eqref{AOdef} or \eqref{IOdef}, reference to the respective distributions would suffice.
  Condition~\eqref{indep2}
 also entails that in general, contrary to the usual setting,
$\Lw(X,Y^{\rm\SSs id})$ is not element of  $\partial{\cal U}^{\rm\SSs SO}(r)$, i.e.; not representable itself as some $\Lw(X,Y^{\rm\SSs re})$
in this neighborhood.
As corresponding (convex) neighborhood we define
\begin{equation}
{\cal U}^{\rm\SSs SO}(r)= \bigcup_{0\leq s \le r} \partial{\cal U}^{\rm\SSs SO}(s)
\end{equation}
%hence the symbol ``$\partial$'' in $\partial{\cal U}^{\rm\SSs SO}$.
Of course, ${\cal U}^{\rm\SSs SO}(r)$ contains $\Lw(X,Y^{\rm\SSs id})$.
In the sequel where clear from the context we drop the superscript ${\rm\Ts SO}$ and the argument $r$.\medskip\\
With this setting we may formulate two typical robust optimization problems:
\paragraph{Minimax-SO problem}
Minimize the maximal MSE on an SO-neighborhood, i.e.; find a measurable reconstruction $f_0$  for $X$ s.t.\
\begin{align}
\quad&\max\nolimits_{{\cal U}}\, \Ew_{\SSs\rm re} |X-f(Y^{\rm\SSs re})|^2 = \min\nolimits_f{}! \label{minmaxSO}
\end{align}
\paragraph{Lemma5-SO problem}
As an analogue to \citet[Lemma~5]{Ha:68}, minimize the MSE in the ideal model but
subject to bound on the bias to be fulfilled on the whole neighborhood, i.e.; find a measurable reconstruction
$f_0$  for $X$ s.t.\
\begin{align}
\quad& \Ew_{\SSs\rm id} |X-f(Y^{\rm\SSs id})|^2 = \min\nolimits_f{}! \quad
   \mbox{s.t.}\;\sup\nolimits_{\cal U}\big|\Ew_{\SSs\rm re} f(Y^{\rm\SSs re})-\Ew X \big|\leq b \label{Lem5SO}
\end{align}

The solution to both problems can be summarized as
\begin{Thm}[Minimax-SO, Lemma5-SO]\label{ThmSO} %\mbox{\hspace{1mm}}\\[-1.5ex]
\begin{enumerate}
\item[(1)]
In this situation, there is a \emph{saddle-point\/} $(f_0, P_0^{Y^{\rm\SSs di}})$ for Problem~\eqref{minmaxSO}
\begin{eqnarray}
%\quad 
f_0(y)\!\!&\!\!:=\!\!&\!\!\Ew X +D(y)w_r(D(y)),\quad w_r(z)=\min\{1, \rho/|z|\} \label{f0def}\\
%\quad 
P_0^{Y^{\rm\SSs di}}(dy)\!\!&:=\!\!&\!\!\Tfrac{1-r}{r} ( \big|D(y)\big|\!/\!\rho\,-1)_{\SSs +}\,\, P^{Y^{\rm\SSs id}}(dy) \label{P0def}
\end{eqnarray}
where $\rho>0$ ensures that $\int \,P_0^{Y^{\rm\SSs di}}(dy)=1$ and
\begin{equation} \label{Ddef}
D(y)=\Ew_{\SSs\rm id}[X|Y=y]-\Ew X
\end{equation}
The value of the minimax risk of Problem~\eqref{minmaxSO} is
\begin{equation} \label{sadvalSO}
\tr \Cov (X) -(1-r)\Ew_{\rm\SSs id}\big[\,|D(Y^{\rm\SSs id})|^2 w_r(Y^{\rm\SSs id})\big]
\end{equation}
\item[(2)] $f_0$ from \eqref{f0def} also is the solution to Problem~\eqref{Lem5SO} for $b=\rho/r$.
\item[(3)] If $\Ew_{\SSs\rm id}[X|Y]$ is linear in $Y$, i.e.; $\Ew_{\SSs\rm id}[X|Y]=MY$ for some matrix $M$, then necessarily
\begin{equation}
M=M^0=\Cov(X,Y)\Var Y^{-}
\end{equation}
or in SSM formulation: $M^0$ is just the classical Kalman gain  and $f_0$ the (one-step) rLS.
\end{enumerate}
\end{Thm}

\subsubsection{Identifications for the SSM context}
Identifying $X$ in model~\eqref{decM} with $\Delta X_t$
and $\pi(y,x)\,\mu(dy)$ with  ${\cal L}(Z_t \Delta X_t +\ve_t)(dy)$,
our ``Bayesian'' Model~\eqref{decM} covers the SSM context. Hence,
if $\Delta X_t$ is normal, (3) applies and rLS is SO-minimax.

\subsubsection{Example for SO-least favorable densities}%{asec5.5}{Example }
To illustrate the result of Theorem~\ref{ThmSO}, we have plotted the
ideal density of $P^{Y^{\rm\SSs id}}$, the (least favorable) contaminated density of $P_0^{Y^{\rm\SSs re}}$,
and the (least favorable) contaminating density of $P_0^{Y^{\rm\SSs di}}$ in Figure~\ref{Fig2.}.

\begin{figure}[!htb]
\begin{center}
\includegraphics[height=7cm, width=10.5cm]{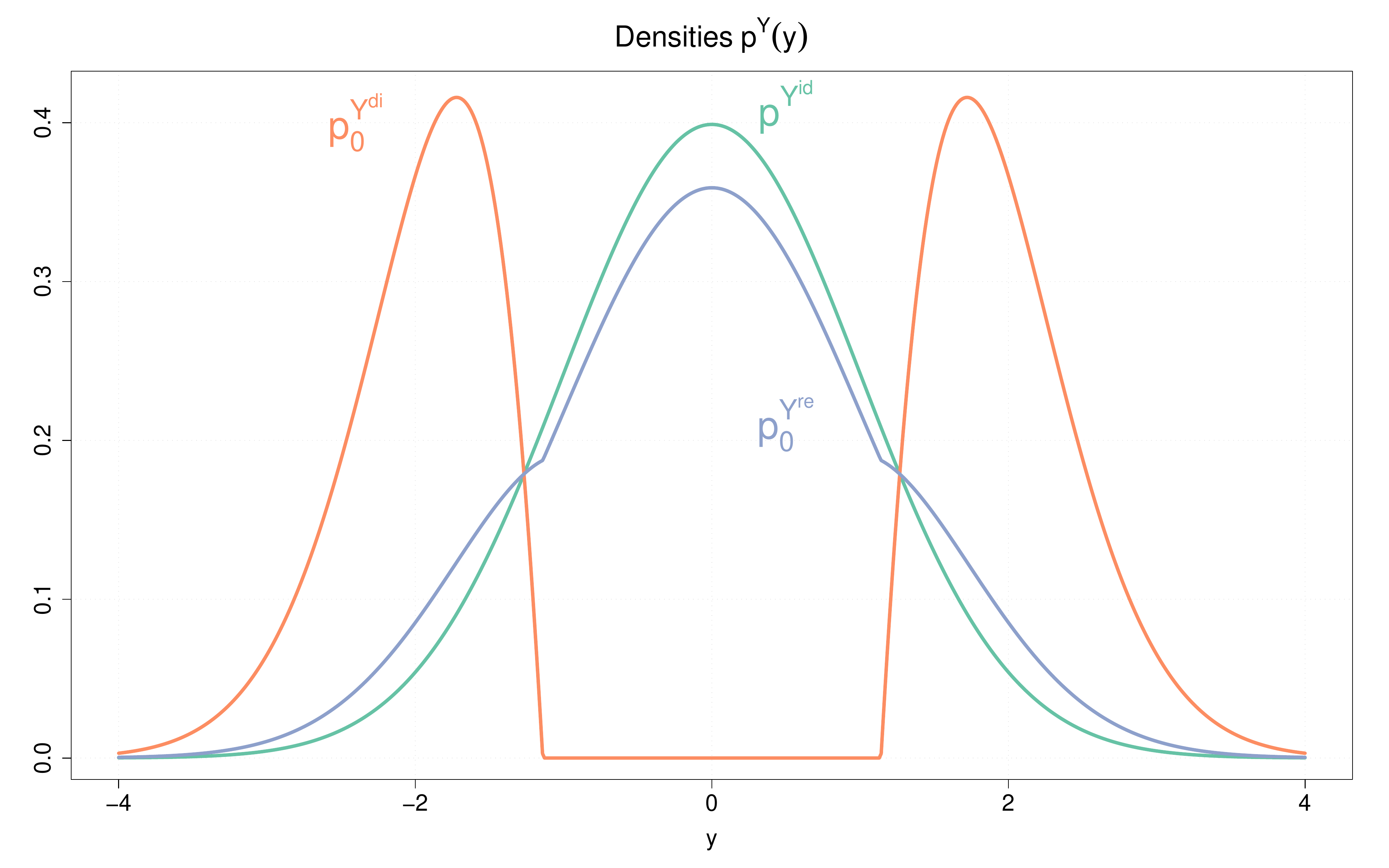}
\caption{\label{Fig2.}Densities of $P^{Y^{\rm\SSs id}}$, $P_0^{Y^{\rm\SSs re}}$, $P_0^{Y^{\rm\SSs di}}$ for
${P^X}={P^{\ve}}={\cal N}(0,1)$, $r=0.1$; note the ``thin'' tails.}
\end{center}
\end{figure}

\begin{Rem}\rm\small  \label{rem33}
  \begin{ABC}
\item Without using this name, SO neighborhoods have already been used by \citet{B:S:93} and \citet{B:P:94},
although only in a one-dim.\ model.
\item Explicit solutions to robust optimization problems in a finite sample setting are rare,
which is why one usually appeals to asymptotics instead.  Important exceptions are
\citet{Hu:68b}, \citet{Hu:Str:73}, and even there, in the former case one is limited to a special loss function
and to one dimension. Our results however are valid in a finite sample context and in whole
generality.
\item Although the structure of our model resembles a location model---interpre\-ting $X$ as a random location parameter---%
our saddle-point differs from the one obtained in \citet{Hu:64}. To see this, let us look at the tails of the least
favorable $P_0^{Y^{\rm\SSs re}}$ assuming a Gaussian model for simplicity:
while in Huber's setting the tails decay as $c e^{-k|x|}$ for some $c,k>0$, in our
setting they decay as $c' |x| e^{-x^2/2}$ so appear even ``less harmful'' than in the location case.
\item Attempts to solve corresponding optimization problems in a (narrow-sense) AO neighborhood are much more
  difficult and only partial results in this context have been obtained in \citet{Dono:78}, \citet{Bi:81},
  and \citet{Bi:Co:83};
  in particular one knows, that in the setup of our example the least favorable $\tilde P^{\ve}=P_0^{\ve^{\rm\SSs di}}$
  must be discrete with only possible accumulation points $\pm \infty$. In addition, existence of a saddle-point
  follows from abstract compactness and continuity arguments, but in order to obtain specific solutions one has to recur
  to numeric approximation techniques as e.g.\ worked out in \citet[sec.~8.3]{Ru:01}; in particular, one obtains redescending
  optimal filters. %; contrary to robust estimation context, this redescending in our filtering context does not pose
%  local-minima problems, because we do not iterate the filter.
\item Redescenders are also used in the ACM filter  by \citet{Ma:Ma:77} which formally translates
the \citet{Hu:64} minimax variance result to this dynamic setting (formally, because of the randomness of
the ``location parameter'' $\Delta X$).
It should be noted though that the least-favorable SO-situation for the ACM then is not in the tails but rather where the
corresponding $\psi$ function takes its maximum in absolute value.
An SO outlier could easily place contaminating mass on this maximum, while this is much harder if not impossible
to achieve in a (narrow-sense) AO situation.
Hence in simulations where we produce ``large'' outliers, the ACM filter tends to outperform the rLS filter,
as these ``large'' outliers are least favorable for the rLS but not for the ACM.
The ``inliers'' producing the least favorable situation for the ACM on the other hand will be much harder
to detect on na\"ive data inspection than ``large'' outliers, in particular in higher dimensions.
 \end{ABC}
\end{Rem}

%-------------------------------------------------------------------------------
\subsection[Back in the original model]{Back in the $\Delta X$ Model for $t>1$}\label{backoindyn}
%-------------------------------------------------------------------------------

So far, in this section, we have ignored the fact that our $X$ in model~\eqref{decM} resp.\ $\Delta X_t$ in the SSM
context will stem from a past which has already used our robustified version of the Kalman filter.
In particular, the law of $\Delta X_t$ (even in the ideal model) is not straightforward and
hence (ideal) conditional expectation appearing in the optimal solution $f_0$ in Theorem~\ref{ThmSO}
in practice are not so easily computable.

\subsubsection{Approaches to go back}
The issue to assess the law $\Delta X_t$ from a non-linear filter  past
is common to other robustifications, and hence there already exist a couple of approaches to deal with it:
\citet{Ma:Ma:77} and \citet{Ma:79}  assume ${\cal L}(\Delta  X_t)$  {\em normal} and propose
using robust location estimators (with redescending $\psi$-function) as alternatives to the
linear correction step. Contradicting this assumption in the rLS case, we have the following proposition

\begin{Prop}\label{nonorm}
  Whenever in one correction step in the $\Delta X_t$ past one has used the rLS-filter, then $\{\Delta  X_t\}$ (as a process) cannot be normally distributed;
  this assertion cannot even hold asymptotically, as long as
  \begin{equation} \label{btbound}
  0<\liminf_t b_t \leq \limsup_t b_t < \infty
  \end{equation}
\end{Prop}

Similar assertions can also be proven for particular $\psi$-functions used in the ACM filter
of \cite{Ma:Ma:77} and \cite{Ma:79}.

\citet{Sch:89} and \citet{Sch:Mi:94} use Taylor-expansions for non-normal ${\cal L}(\Delta  X_t)$;
doing so they end up with stochastic error terms but do not give an indication as to uniform integrability.
Hence it is not clear whether the approximation stays valid after integration. More importantly, at time instance $t$,
they come up with a bank of (at least $t$) Kalman--filters which is not operational.

\citet{B:S:93} work with the exact ${\cal L}(\Delta  X_t)$ and hence have to split up the integration according
to the the history of outlier occurrences which yields $2^t$ different terms---which  is not operational
either.

\begin{Rem}\rm\small   \label{RemNonRec}
One of the features of the ideal Gaussian model is that
$\Ew_{\rm\SSs id}[\Delta X_t|Y_{1:t}]$ is Markovian in the sense that
$\Ew_{\rm\SSs id}[\Delta X_t|Y_{1:t}]=\Ew_{\rm\SSs id}[\Delta X_t|\Delta Y_{t}]$
hence only depends on the one value of $\Delta Y_t$.
When using bounded correction
steps, however, this property gets lost, hence
the restriction to strictly recursive procedures as is the rLS filter is a real restriction.
\end{Rem}

Theorem~\ref{ThmSO} does not make any normality assumptions, but in assertion~(3), we have seen that
the rLS would result optimal once we can show that  $\Ew_{\rm\SSs id}[\Delta  X_t|\Delta Y_t]$
for $\Delta X$ stemming from an rLS past is {\em linear}. This leads to the question: % \\%[3.5ex]
%
%\begin{center}%
\textbf{When is \boldmath$\Ew_{\rm\SSs id}[\Delta  X|\Delta Y]$ linear?} %
%\ifx\ITWMdef\undefined%
%\medskip\\%
%\fi%
%\end{center}
%\\
%\noindent
Omitting time indices $t$, the answer is
\begin{Prop}\label{nolin}
Assume $\rk(\EM_p-MZ)=p$, $p=q$ and $\rk Z=p$, and that
\begin{equation} \label{normeps}
{\cal L}_{\rm \SSs id}(\ve) ={\cal N}_q(0,V),\qquad \ve \mbox{ independent of } \Delta X
\end{equation}
Then $\Ew_{\rm\SSs id}[\Delta  X|\Delta Y]$ is linear
%\begin{itemize}
%\item[]
\begin{eqnarray}
%\quad\;\;
\hspace{-4em}&\iff&\!\!\!{\cal L}_{\rm \SSs id}(\Delta  X) \quad\mbox{is normal} \label{linnorm}\\
%\quad\;\;
\hspace{-4em}&\iff&\!\!\! M_3(e):=\Ew_{\rm \SSs id}\big[\big(e^\tau (\Delta X - \Ew[\Delta X|\Delta Y])\big)^3\,\big|\,\Delta Y\!=\!y\big]=0 \;\;\;\forall\,e \in \R^p  \label{3mom}
\end{eqnarray}
%\end{itemize}
\end{Prop}
\begin{Rem}\rm\small    \begin{ABC}
  \item Assumption $\rk(\EM_p-MZ)=p$ is fulfilled in most situations; otherwise there is a one-dimensional projection
  of the filter error that is $0$ almost sure.
  \item For $Z$ non-invertible, in particular for $p\not=q$, equivalence~\eqref{linnorm} still holds, if we require
\begin{equation} \label{PiNorm}
{\cal L}_{\rm \SSs id}(\Pi \Delta X) ={\cal N}_p(0,\Pi\Sigma\Pi),\qquad  \Pi \Delta X \mbox{ independent of }\bar\Pi\Delta X
\end{equation}
where $\Pi$ is the projector onto $\ker Z$ and $\bar \Pi=\EM_p-\Pi$. In fact we prove Proposition~\ref{nolin}
in this more general case. Assumption~\eqref{PiNorm} is needed, as $\Pi\Delta X$ is invisible for $\Delta Y$.
  \item Equivalence~\eqref{linnorm} together with Proposition~\ref{nonorm}
  shows that, stemming from an rLS-past, rLS can only be SO-optimal in the very first time step.
  \item Simulations however show that rLS gives very reasonable results. So in fact
  we could/should be close to an ideal linear conditional expectation.
  ``Closeness'' to linearity could be quantified by the second derivative %\linebreak[4]
  $\partial^2/\partial y^2  \Ew_{\rm\SSs id}[\Delta  X|\Delta Y=y]$, which in fact
  leads us to expression~\eqref{3mom}.
  \item Equivalence~\eqref{3mom}, i.e.; conditional unskewedness of $\Delta X$, is somewhat surprising, as it
  seems  much weaker than normality of the prediction error.
  \item Condition~\eqref{normeps} could be relaxed to $\ve~\sim P$, $P$ some infinitely divisible distribution, and
  the normality assumption in \eqref{PiNorm} be dropped.
  Equivalence~\eqref{linnorm} would then become: For each $M\in\R^{p\times q}$ there can be at most one
  distribution $Q=Q(M,P)$ on $\B^p$, such that
  $\Ew[\Delta X|\Delta Y]=M\Delta Y$ for ${\cal L}(\bar \Pi \Delta X)=Q$; for $p=q=1$ and $Z\not=0$,
  there always is such a $Q$; see \citet[Thm.~1.3.1]{Ru:01}.
  \end{ABC}
\end{Rem}
\subsubsection{A test for linearity}
In particle filter context where you simulate many stochastically independent filter realizations in parallel,
Proposition~\ref{nolin} suggests the following test for linearity/normality:

\begin{Prop} \label{testprop}
  Let $\Delta X_i^\natural$, $i=1,\ldots,n$ be an i.i.d.\ sample from ${\cal L}(\Delta X_t)$, the law of
  the prediction errors of some filter at time $t$; let $\Sigma=\Cov(\Delta X_t) $,
  $\sigma^2$ its maximal eigenvalue and $e$ a corresponding eigenvector (of norm $1$);
  let $\hat \Sigma_n$, $\hat \sigma_n^2$, and $\hat e_n$ the corresponding empirical
  counter parts (all assumed consistent). Define the test statistic
$%  \begin{equation}
  T_n= \frac{1}{n} \sum_{i=1}^n (\hat e_n ^\tau \Delta X_i^\natural)^3 %\label{statdef}
$. %  \end{equation}
  Then under normality of ${\cal L}(\Delta X_t)$,
  \begin{equation}
  \sqrt{n}\, T_n \wto {\cal N}(0, 15 \sigma^6) \label{asynorm}
  \end{equation}
  and the test
  \begin{equation}
  \Jc(|T_n| > \sqrt{15/n} \, \hat \sigma_n^3 u_{\alpha/2}) \label{Gausstest}
  \end{equation}
  for $u_{\alpha}$ the upper $\alpha$-quantile of ${\cal N}(0,1)$
  is asymptotically most powerful among all unbiased level-$\alpha$-tests for testing
  \begin{equation}
  \mbox{$H_0$}\colon\quad \sup_{|e|=1} M_3(e)=0 \qquad \mbox{vs.} \qquad \mbox{$H_1$}\colon\quad \sup_{|e|=1} |M_3(e)|>0
  \end{equation}
\end{Prop}
%-------------------------------------------------------------------------------
\subsection{Way out: eSO-Neighborhoods} \label{wayoutsec}
%-------------------------------------------------------------------------------
One explanation for the good empirical findings for the rLS is given by a further
 extension of  the original SO-neighborhoods---the {\it e}xtended \textit{SO} or \textit{${\rm eSO}$--model\/}:
In this model, we  also allow for model deviations in $X$, i.e.; we assume a
realistic $(X^{\rm\SSs re},Y^{\rm\SSs re})$ according to
\begin{equation}
(X^{\rm\SSs re},Y^{\rm\SSs re}):=(1-U)(X^{\rm\SSs id},Y^{\rm\SSs id})+U(X^{\rm\SSs di},Y^{\rm\SSs di}) \label{eso1}
\end{equation}
for $X^{\rm\SSs id}\sim P^{X^{\rm\SSs id}}$, $Y^{\rm\SSs id}$ according to equation~\eqref{decM}, $X^{\rm\SSs di}\sim P^{X^{\rm\SSs di}}$,
$Y^{\rm\SSs di}\sim P^{Y^{\rm\SSs di}}$,
$U\sim{\rm Bin}(1,r_{\rm\SSs eSO})$,  where
\begin{equation} \label{indep3}
 U \mbox{ and }(X^{\rm\SSs id},Y^{\rm\SSs id})\;\mbox{  independent as well as (mutually) }  U, X^{\rm\SSs di},Y^{\rm\SSs di}
\end{equation}
and the joint law $P^{X^{\rm\SSs id},Y^{\rm\SSs id}}$ and the radius $r=r_{\rm\SSs eSO}$ are known, while $P^{X^{\rm\SSs di}}, P^{Y^{\rm\SSs di}}$ are arbitrary, unknown
  and uncontrollable; however, we assume that
\begin{equation}\label{esogl}
\Ew_{\rm\SSs di} X^{\rm\SSs di} =\Ew_{\rm\SSs id} X^{\rm\SSs id}, \qquad \Ew_{\rm\SSs di} |X^{\rm\SSs di}|^2 \leq G
\end{equation}
for some known $0<\Ew_{\rm\SSs id} |X^{\rm\SSs id}|^2 \leq G<\infty$, and accordingly define
\begin{equation}
\!\!\!\! {\cal U}^{{\rm\SSs eSO}}(r):=\!\!\!\! \bigcup_{0\leq s \le r}\!\! \partial{\cal U}^{\rm\SSs eSO}(s),\quad
\partial{\cal U}^{\rm\SSs eSO}(r) :=  \{\; {\cal L}(X^{\rm\SSs re},Y^{\rm\SSs re})\;\mbox{acc. to \eqref{eso1}--\eqref{esogl}}\; \}
\end{equation}
\begin{Rem}\rm\small    At first glance, moment condition~\eqref{esogl} seems to violate (distributional) robustness; however, this condition
  has not been introduced to induce a higher degree of robustness, but rather to extend the applicability of Theorem~\ref{ThmSO}.
\end{Rem}
\begin{Thm}[minimax-eSO]\label{ThmeSO}
The pair $(f_0, P_0^{Y^{\rm\SSs di}})$, optimal in the Minimax-SO-problem to radius $r_{\rm\SSs SO}=r$
from Theorem~\ref{ThmSO}, extended to $\big(f_0, P_0^{Y^{\rm\SSs di}} \otimes P_0^{X^{\rm\SSs di}}\big)$ for any
$P_0^{X^{\rm\SSs di}}$ such that  $\Ew_{\rm\SSs di} |X^{\rm\SSs di}|^2 = G$, remains a saddle-point in the corresponding Minimax-Problem on the ${\rm eSO}$-neighborhood
${\cal U}^{\rm\SSs eSO}$ to the
same radius $r$---no matter what bound $G$ in equation~\eqref{esogl} holds. The value of the minimax risk is
\begin{equation}
\tr \Cov_{\rm\SSs id} X^{\rm\SSs id} + r (G-\Ew_{\rm\SSs id} |X^{\rm\SSs id}|^2) -(1-r)\Ew_{\rm\SSs id}\big[\,|D(Y^{\rm\SSs id})|^2 w_r(Y^{\rm\SSs id})\big]
\label{sadvaleSO}
\end{equation}
\end{Thm}
As an application of Theorem~\ref{ThmeSO}, we now invoke a coupling idea:
In the Gaussian setup, i.e.; we assume \eqref{normStart}, we no longer regard
the (SO--) saddle-point solution to an ${\cal U}(r)$-neighborhood
around ${\cal L}(\Delta  X)$ stemming from an rLS-past, but use
Theorem~\ref{ThmeSO}  as follows:
\begin{Prop}\label{esoprop}
Assume that for each time $t$ there is a
{\em (fictive)} random variable $\Delta  X^{\cal N}\sim{\cal N}_p(0,\Sigma)$
such that  $\Delta  X_t^{\rm\SSs rLS}$ stemming from an rLS-past can be considered an $X^{\rm\SSs di}$ in
the corresponding ${\rm  eSO}$-neighborhood around $\Delta  X^{\cal N}$
with  radius $r$.
Then, rLS is exactly minimax for each time $t$.
\end{Prop}

\begin{Rem}\rm\small    \begin{ABC}
 %   \item Proposition~\ref{esoprop} gives an explanation for the good empirical results obtained with the rLS filter,
 %   compare [BSPANGL-REF].
    \item Existence of $\Delta  X^{\cal N}\sim{\cal N}_p(0,\Sigma)$ in a general setting is not yet proved.
    To this end one has to show moment condition~\eqref{esogl} and that
    \begin{equation} \label{esocond}
    \sup\nolimits_{\lambda} \big(p^{\Delta X_t^{\cal N}}\,/\, p^{\Delta  X_t}\big) \ge 1-r
    \end{equation}
    where $p^{\Delta X_t^{\cal N}}$, $p^{\Delta X_t}$ are the corresponding Lebesgue densities and  $\sup\nolimits_{\lambda}$ is the corresponding
    essential supremum w.r.t.\ Lebesgue measure in the respective dimension. Clearly condition~\eqref{esocond} is
    the difficulty, while condition~\eqref{esogl} is not hard to fulfill---we only need to check that $\Ew_{\rm\SSs id} \Delta X_t=0$,
    which for the rLS follows from symmetry of the distributions in the ideal model, and that
    the second moment is bounded---which also clearly holds.
    \item As to the choice of covariance~$\Sigma$ for $\Delta X_t^{\cal N}$, we have two candidates:
    $\Sigma = \Cov \Delta X_t^{\rm\SSs rLS}$ and $\Sigma = \Sigma_{t|t-1}$ from the classical Kalman filter.
    While the former takes up the actual error covariances, the latter is much easier to compute. In
    our numerical examples in \citet{Ru:01}, we could not find any significant advantages for the former
    in terms of precision and hence propose the latter for computational reasons.
    \item For $p=1$, \eqref{esocond} could be checked numerically in a number of models,
    cf.\  \citet[Table~8.1]{Ru:01}.
    For  $p>1$, particle filter techniques should be helpful.
  \end{ABC}
\end{Rem}
%
% ------------------------------------------------------------------------
\section{IO-optimality} \label{IO'sec}
% ------------------------------------------------------------------------
In this section, we translate the preceding optimality results to the IO situation.
We have already noted that in this case, instead of attenuating (the influence of)
a dubious observation we would rather want to follow an IO outlier as fast as
possible. It is well-known that the Kalman filter tends to be too inert for this
task and faster tracking filters are needed. To do so, let us go  back to our ``Bayesian'' model~\eqref{decM}
but now we specify the transition densities $\pi(y,x)$
to come from an observation $Y$ which is built up additively as
\begin{equation} \label{simpAdd}
Y=X+\ve
\end{equation}
Equation~\eqref{simpAdd} reveals a remarkable symmetry of $X$ and $\ve$ which
we are going to exploit now: Apparently
\begin{equation} \label{simpEx}
\Ew[X|Y] = Y-\Ew[\ve|Y]
\end{equation}
This is helpful if we are now assuming that $\ve$ will be ideally distributed,
and instead the states $X$ get corrupted. To this end, we retain the SO-model
from the preceding sections, i.e., $Y^{\rm\SSs id}$ will be replaced from time
to time by $Y^{\rm\SSs di}$. Contrary to the AO formulation however, we now assume that
this replacement by $Y^{\rm\SSs di}$ reflects a corresponding change in $X$, as
we now want to track the distorted signal.
As a consequence this gives the following IO-version of the minimax problem (where the only visible
difference is the superscript ``${\rm\Ts re}$'' for $X$).
\begin{align}
\quad&\max\nolimits_{{\cal U}}\, \Ew_{\SSs\rm re} |X^{\rm\SSs re}-f(Y^{\rm\SSs re})|^2 = \min\nolimits_f{}! \label{minmaxIOSO}
\end{align}
But, using $X^{\rm\SSs re}=Y^{\rm\SSs re}-\ve$, and setting $\tilde f(y)=y-f(y)$ we obtain the equivalent
formulation
\begin{align}
\quad&\max\nolimits_{{\cal U}}\, \Ew_{\SSs\rm re} |\ve-\tilde f(Y^{\rm\SSs re})|^2 = \min\nolimits_{\tilde f}{}! \label{minmaxIOSO2}
\end{align}
and we are back in the situation of subsection~\eqref{1stpOpt} with the respective r\^oles of $X$ and $\ve$ interchanged.
That is; the corresponding theorems translate word by word. Skipping the Lemma~5 solution we obtain

\begin{Thm}[Minimax-IO]\label{ThmISO}%\mbox{\hspace{1mm}}\\[-1.5ex]
\begin{enumerate}
\item[(1)']
In this situation, there is a \emph{saddle-point\/} $(f_1, P_1^{Y^{\rm\SSs di}})$ for Problem~\eqref{minmaxIOSO}
\begin{eqnarray}
f_1(y)&:=&y- \tilde D(y) \min\{1, \tilde \rho/\big|\tilde D(y)\big|\}\\
P_1^{Y^{\rm\SSs di}}(dy)&:=&\Tfrac{1-r}{r} ( \big|\tilde D(y)\big|\, \big/ \tilde \rho\,-1)_{\SSs +}\,\, P^{Y^{\rm\SSs id}}(dy)
\end{eqnarray}
where $\tilde \rho>0$ ensures that $\int \,P_1^{Y^{\rm\SSs di}}(dy)=1$ and
\begin{equation} \label{Dy2}
\tilde D(y)=y-\Ew_{\SSs\rm id}[X|Y=y]
\end{equation}
\item[(3)']
If $\Ew_{\SSs\rm id}[X|Y]$ is linear in $Y$, i.e.; $\Ew_{\SSs\rm id}[X|Y]=MY$ for some matrix $M$, then necessarily
\begin{equation}
M=M^0=\Cov(X,Y)\Var Y^{-}
\end{equation}
---or in the SSM formulation: $M^0$ is just the classical Kalman gain  and $f_1$ the (one-step) rLS.IO defined below.
\end{enumerate}
\end{Thm}
 Note that contrary to Theorem~\ref{ThmSO} where $\Ew X$ need not be $0$, here
$\Ew \ve = 0$, which simplifies the definition of $\tilde D$ in \eqref{Dy2}.
Details on how to use this for a corresponding IO-robust variant of rLS are given in \citet{Ru:10}.
\section{Conclusion and Outlook} \label{conclusion}
In the extremely flexible class of dynamic models consisting in SSMs we
were able to obtain optimality results for filtering. In this generality this is a novelty.
We stress the fact that our filters are non-iterative, recursive,
hence fast, and valid for higher dimensions. %\\

So far, we have not said much about the implementation of these filters.
rLS.AO was originally implemented to {\tt XploRe}, compare \citet{Ru:00}.
In an ongoing project with Bernhard Spangl, BOKU, Vienna,
and Irina~Ursachi (ITWM), we are about to implement the rLS filter to {\sf R},
(\citet{R09}), more specifically to an {\sf R}-package {\tt robKalman}, the
development of which is done under {\tt r-forge} project
\hreft{https://r-forge.r-project.org/projects/robkalman/},
(\citet{Rforge}). Under this address you will also find a preliminary version
available for download.

In an extra paper, which for the moment is available as technical report, \citet{Ru:10},
we also check the properties of our filters at simulations and discuss
the extension of these optimally-robust filters to a filter that combines the two types
(for system-endogenous and -exogenous outlier situation). This hybrid filter is capable to
  treat (wide-sense) IO's and AO's simultaneously---albeit with minor delay.

% ------------------------------------------------------------------------
\section{Proofs} \label{proofs}
% ------------------------------------------------------------------------
\paragraph{Proof to Lemma~\ref{Lem223}}
We use the fact that for $0\leq a,b,c,d$, $
%\begin{equation}
(a+b)/(c+d) \leq \max(a/c,b/d)$. Hence
%\nonumber
%\end{equation}We start by showing that
  \begin{equation} \label{Lem223a}
  \rho_0(s)\le \max\{A_s/A_{r_l},B_s/B_{r_u}\}
  \end{equation}
Equation~\eqref{deltakrit2} shows that $b(r)$ is (strictly) decreasing in $r$ (for $r>0$) from $\infty$ to $0$.
Hence $A_r$ is increasing in $r$, and $B_r$ decreasing, $B_r$ from $\infty$ to $0$.
By dominated convergence $b(r)$, and hence $A_r$ and $B_r$ are continuous in $r$.
Thus  existence of $\bar r_0$ follows. For $r_u=1$, one argues letting $r_n\in[0,1)$ tend to $1$.
To show equality in \eqref{Lem223a}, we parallel \citet[Lemma~2.2.3]{Ko05}, and first show that for $r\ge s$, $s$ fixed, $\rho(r,s)$ is increasing and
correspondingly, for $r\le s$, $s$ fixed, decreasing, which entails \eqref{Lem223c}: Let $0\leq s<r_1<r_2\leq 1$. Then by monotony of $A_r$, $B_r$,
$(A_sB_s^{-1}+r_1)^{-1} \ge (A_{r_1}B_{r_1}^{-1}+r_1)^{-1}$; multiplying this inequality with $(r_2-r_1)$, we get
$%\begin{equation}
 (r_2-r_1)B_s(A_s+ r_1 B_s)^{-1} \ge (r_2-r_1)B_{r_1}(A_{r_1}+r_1B_{r_1})^{-1}%\nonumber
$. %\end{equation}
Now, due to optimality of   $A_{r}+rB_{r}$ for radius $r$,
\begin{eqnarray*}
0&\leq&\frac{(r_2-r_1)B_s}{A_s+ r_1 B_s} - \frac{(r_2-r_1)B_{r_1} + A_{r_2}+r_2B_{r_2} -A_{r_1}-r_2B_{r_1}}{A_{r_1}+r_1B_{r_1}} =\\
&=&(r_2-r_1)B_s(A_s+ r_1 B_s)^{-1} - \big(A_{r_2}+r_2B_{r_2}\big)(A_{r_1}+r_1B_{r_1})^{-1} +1\nonumber
\end{eqnarray*}
Multiplying with $(A_{s}+r_1B_{s})/(A_{r_2}+r_2B_{r_2})$, we obtain indeed
$\rho(r_2,s) \ge \rho(r_1,s)$, and similarly for $0\ge s>r_1>r_2\ge 1$. Next, for $\tilde r_0$ least favorable,
we show that for $r$ fixed, and $s\ge r$, $\rho(r,s)$ is increasing and
 correspondingly, for $s\le r$, decreasing: Let $0\le r<r_1<r_2\le 1$. Then, due to optimality of $A_{r_1}-r_1 B_{r_1}$,
\begin{eqnarray*}
&&A_{r_2}+rB_{r_2}-A_{r_1}+rB_{r_1}=\\
&&\quad=(r_1-r)(B_{r_1}-B_{r_2})+A_{r_2}+r_1 B_{r_2} -A_{r_1}-r_1 B_{r_1}\ge 0
\end{eqnarray*}
and similarly for $0\ge r>r_1>r_2\ge 1$.
For the last assertion, note that by  \eqref{deltakrit2}, $b(1)=0$, hence $B_1=0$.
Hence $\max\big\{A_s/A_{r_l},B_s/B_{1}\big\}=\infty$ for  $s<1$, while for $s=1$, we get
 $\rho_0(1)=\max\{A_1/A_{r_l},1\}=1$. \hfill\qed
%\end{proof}
%
%
\paragraph{Proof to Theorem~\ref{ThmSO}}
%\begin{enumerate}
%  \item[(1)]
(1) Let us solve $\max_{\partial{\cal U}}\min_f{} [\ldots]$ first, which amounts to $\min_{\partial{\cal U}} \Ew_{\SSs \rm re}[\big|\Ew_{\SSs \rm re}[X|Y^{\rm\SSs re}]\big|^2]$.
For fixed element $P^{Y^{\rm\SSs di}}$ assume w.l.o.g.\ that $\mu\gg P^{Y^{\rm\SSs di}}$
for $\mu$ from \eqref{decM}---otherwise we replace $\mu$ by $\mu + P^{Y^{\rm\SSs di}}$; this gives us a $\mu$-density
$q(y)$ of $P^{Y^{\rm\SSs di}}$.
Determining the joint (real) law $P^{X,Y^{\rm \SSs re}}(dx,dy)$ as
\begin{equation} \label{SOdistr}
P(X \!\in\! A, Y^{\rm \SSs re} \!\in\! B) =\!\! \int \!\! \Jc_A(x)\Jc_B(y) [(1\!-\!r) \pi(y,x) +r q(y)]\, P^X(dx)\,\mu(dy)
\end{equation}
we deduce that $\mu(dy)$-a.e.
\begin{equation} \label{SOEw}
\Ew_{\rm\SSs re}[X|Y^{\rm\SSs re}\!\!=\!y]=\frac{r q(y)\!\Ew X \!+\!(1\!-\!r)p^{Y^{\rm\SSs id}}(y)\Ew_{\SSs \rm id}[X|Y]}%
{r q(y)+(1-r)p^{Y^{\rm\SSs id}}(y)}\!=:\!\frac{a_1 q(y)\!+\! a_2(y)}{a_3q(y)\!+\!a_4(y)}
\end{equation}
Hence we have to minimize
\begin{equation}
F(q):= \int \frac{|a_1 q(y)+ a_2(y)|^2}{a_3q(y)+a_4(y)}\,\,\mu(dy)\nonumber
\end{equation}
 in $M_0=\{q\in L_1(\mu)\,|\; q\geq 0,\; \int q\,d\mu=1\}$.
To this end, we note that $F$ is convex on the non-void, convex cone $M=\{q  \in L_1(\mu)\,|\; q\geq 0\}$ %---in fact even on $M'=\{q  \in L_1(\mu)\,|\; q\geq -a_4(y)/a_3\}$---%
so, for some $\tilde \rho\ge 0$, we may consider the Lagrangian
\begin{equation}
L_{ \tilde \rho}(q):=F(q) + \tilde \rho \int q\,d\mu
\end{equation} for some positive Lagrange multiplier $\tilde \rho$.
Pointwise minimization in $y$ of $L_{ \tilde \rho}(q)$ gives % without any restriction gives us the form
\begin{equation}
%q^0_s(y)=\Tfrac{1-r}{r} ( \big|D(y)\big|\big/s\,-1)\,\, p^{Y}(y)\nonumber
%\hat
q_s(y)=%\big(q^0_s(y)\big)_{\SSs +}=
\Tfrac{1-r}{r} ( \big|D(y)\big|\big/s\,-1)_{\SSs +}\,\, p^{Y}(y)\nonumber
\end{equation}
for some constant $s=s( \tilde \rho)=(\,|\Ew X|^2 + \tilde \rho/r)^{1/2}$,
%and the restriction to $M$ leads to
%\begin{equation}
%\hat q_s(y)=\big(q^0_s(y)\big)_{\SSs +}=\Tfrac{1-r}{r} ( \big|D(y)\big|\big/s\,-1)_{\SSs +}\,\, p^{Y}(y)\nonumber
%\end{equation}
Pointwise in $y$, $\hat q_s$ is antitone and continuous in $s\ge 0$ and $\lim_{s\to 0[\infty]}q_s(y)=\infty[0]$, hence by
monotone convergence,
\begin{equation}
H(s)=\int \hat q_s(y) \,\mu(dy)\nonumber
\end{equation}
 too, is antitone and continuous and $\lim_{s\to 0[\infty]}H(s)=\infty[0]$. So by continuity, there is some $\rho \in (0,\infty)$ with $H(\rho)=1$.
On $M_0$, $\int q\,d\mu =1$, but $\hat q_\rho=q_{s=\rho}\in M_0$ and is optimal on $M\supset M_0$ hence it also minimizes $F$ on $M_0$.
In particular, we get representation \eqref{P0def} and note that, independently from the choice of $\mu$,
the least favorable $P_0^{Y^{\rm\SSs di}}$ is dominated according to $P_0^{Y^{\rm\SSs di}}\ll P^{Y^{\rm\SSs id}}$, i.e.;
non-dominated $P^{Y^{\rm\SSs di}}$ are even easier to deal with.

As next step we %return to the minmax problem, i.e.; $\min_f{} \max_{\partial{\cal U}}[\ldots]$ and
show that
\begin{equation} \label{minmax=maxmin}
\max\nolimits_{\partial{\cal U}}\min\nolimits_f{} [\ldots] = \min\nolimits_f{} \max\nolimits_{\partial{\cal U}}[\ldots]
\end{equation}
To this end we first verify \eqref{f0def} determining $f_0(y)$ as $f_0(y)=\Ew_{\SSs {\rm re};\hat P}[X|Y^{\rm\SSs re}=y]$.
Writing a sub/superscript
``${{\rm re;}\,P}$'' for evaluation under the situation generated by  $P=P^{Y^{\rm\SSs di}}$
and $\hat P$ for $P_0^{Y^{\rm\SSs di}}$, we obtain the the risk for general $P$ as
\begin{eqnarray}
{\rm MSE}_{\SSs{{\rm re;}\,P}}[f_0(Y^{\SSs {\rm re},\,P})]&=&
(1-r)\Ew_{\rm\SSs id}\big|X-f_0(Y^{\rm\SSs id})\big|^2+ r \tr\Cov X+\nonumber\\
&&\quad  +
r\,\Ew_P \min(|D(Y^{\SSs {\rm di;},q})|^2,\rho^2) \label{9.7}
\end{eqnarray}
This is maximal for any $P$ that is concentrated on the set $\big\{\,|D(Y^{\SSs {\rm di;},q})|>\rho\,\big\}$,
which is true for $\hat P$. Hence \eqref{minmax=maxmin} follows, as for any contaminating $P$
\begin{equation}
{\rm MSE}_{\SSs{{\rm re;}\,P}}[f_0(Y^{\SSs{{\rm re;}\,P}}] \le {\rm MSE}_{\SSs{{\rm re;}\,\hat P}}[f_0(Y^{\SSs{{\rm re;}\,\hat P}})]\nonumber
\end{equation}

Finally, we pass over from $\partial{\cal U}$ to ${\cal U}$: Let $f_r$, $\hat P_r$
 denote the components of the saddle-point for $\partial {\cal U}(r)$, as well as $\rho(r)$ the corresponding Lagrange multiplier
and $w_r$ the corresponding weight, i.e., $w_r=w_r(y)=\min(1, {\rho(r)}\,/\,{|D(y)|})$.
 Let $R(f,P,r)$ be the MSE of procedure $f$ at the {\rm SO} model
$\partial {\cal U}(r)$ with contaminating $P^{Y^{\rm \SSs di}}=P$.
As can be seen from \eqref{P0def}, $\rho(r)$ is antitone in $r$; in particular, as $\hat P_r$ is concentrated
on $\{|D(Y)|\ge \rho(r)\}$ which for $r\leq s$ is a subset of $ \{|D(Y)|\ge \rho(s)\}$, we obtain
$$
R(f_s,\hat P_s,s )=R(f_s,\hat P_r,s )\qquad\mbox{for}\;r\leq s
$$
Note that $R(f_s,P,0 )=R(f_s,Q,0 )$ for all $P,Q$---hence passage to $\tilde R(f_s,P,r )= R(f_s,P,r )-R(f_s,P,0 )$
is helpful---and that
\begin{equation} \label{varsplit}
\tr \Cov X= \Ew_{\rm\SSs id} \Big[\tr \Cov_{\rm\SSs id}[X|Y^{\rm\SSs id}]+ |D(Y^{\rm\SSs id})|^2 \Big]
\end{equation}
Abbreviate $\bar w_s(Y^{\rm\SSs id})=1-\big(1-w_s(Y^{\rm\SSs id})\big)^2\ge 0$ to see that
\begin{eqnarray}
\hspace{-2em}&&\tilde R(f_s,P,r )= %\nonumber\\
%&=&
r \Big\{\Ew_{\rm\SSs id}\Big[|D(Y^{\rm\SSs id})|^2 \bar w_s(Y^{\rm\SSs id})\Big] +%\\
\Ew_{P}\min(|D(Y^{\rm\SSs id})|,\rho(s))^2\,\Big\} \leq \nonumber\\
\hspace{-2em}&&\leq r \Big\{\Ew_{\rm\SSs id}\Big[|D(Y^{\rm\SSs id})|^2 \bar w_s(Y^{\rm\SSs id})\Big] + \rho(s)^2\,\Big\}=
\tilde R(f_s,\hat P_r,r )<%\nonumber\\
%&<&
\tilde R(f_s,\hat P_s,s )\nonumber
\end{eqnarray}
Hence the saddle-point extends to ${\cal U}(r)$; in particular the maximal risk is never attained in
the interior ${\cal U}(r)\setminus \partial {\cal U}(r)$. \eqref{sadvalSO} follows by plugging in the results.
%Now, as can be seen from \eqref{P0def}, $\rho=\rho(r)$ is antitone in $r$. Hence the optimal risk is increasing in $r$, so that indeed
%we have a saddle-point valid on ${\cal U} \times \{f \,|\,f \;\mbox{measurable w.r.t.\ $Y^{\rm\SSs re}$}\,\}$.
%
%---------------------------------------------------------------------------------------------------------------------------------------
%
%  \item[(2)]

  (2) Let $\tilde f(Y)=f(Y)-\Ew X $, and $X^0=X-\Ew X$; then  \eqref{Lem5SO} becomes
  \begin{align}
\quad& \Ew_{\SSs\rm id} |X^0 - \tilde f(Y)|^2 = \min\nolimits_{\tilde f}{}! \quad
   \mbox{s.t.}\;\sup\nolimits_{\cal U}\big|\Ew_{\SSs\rm re} \tilde f(Y^{\rm\SSs re})\big|\leq b \label{Lem5SO2}
\end{align}
  The assertion follows upon noting that $\sup_{\cal U}|\Ew_{\rm\SSs re} \tilde f|=\sup |\tilde f|$ (to be shown just as in \citet[chap.~5]{Ri:94})
  and writing
  $$\Ew_{\SSs\rm id} |X^0 - \tilde f(Y)|^2=\Ew_{\SSs\rm id}\Big[\Ew[ |X^0 - \tilde f(Y)|^2\,\Big|\,Y]\Big]$$
  ---minimize the inner expectation subject to $\big|\tilde f(Y^{\rm\SSs re})\big|\leq b$ pointwise in $Y$.
%
%---------------------------------------------------------------------------------------------------------------------------------------
%
%  \item[(3)]

(3) If $\Ew_{\SSs\rm id}[X|Y]$ is linear in $Y$, the corresponding optimal matrix $M^0$ is just the respective
Fourier coefficient, i.e.; $\Cov(X,Y)\Var Y^{-}$.
We have already recalled that the classical Kalman filter is optimal among all linear filters;
hence the corresponding Kalman gain $M^0$ is then the optimal linear transformation in the SSM context.
%\end{enumerate}
\hfill\qed

\begin{Rem}\rm\small  \begin{ABC}
  \item \citet{B:S:93} proceed similarly for their result.  However, they invoke a minimax result
  by \citet{Fer:67} which in our infinite dimensional setting is not applicable. Also their setting is
  restricted to one dimension, and they assume Lebesgue densities right away---also in the contaminated
  situation.  In particular, they do not realize the connection to the exact conditional mean present in equation~\eqref{Ddef}.
  \item For an alternative proof, see \citet[pp.156--163]{Ru:01}: It uses \citet[App.~B]{Ri:94},
  showing existence of Lagrange multipliers in (1) by abstract compactness and continuity arguments.
 \item The fact that the solutions to Problems~\eqref{minmaxSO} and \eqref{Lem5SO} coincide parallels the
 situation in the estimation problem for a one-dimensional location parameter.
\end{ABC}
\end{Rem}
\paragraph{Proof to Proposition~\ref{nonorm}}
  Recall that by the Cram\'er-L\'evy Theorem (cf.\  \citet[Thm.~1, p.~525]{Fe2:71}) the sum of two independent
  random variables has Gaussian distribution iff each summand is Gaussian. This can easily be translated into a
  corresponding   asymptotic statement, cf.\  \citet[Prop.~A.2.4]{Ru:01}, i.e.;
  the sum of two independent
  random variables converges weakly to a Gaussian distribution iff each summand converges weakly to a Gaussian
  distribution.
  We first consider (for fixed $t$, omitted from notation where clear) the filter error,
  \begin{equation}
  \widetilde {\Delta X}:=X_t-X_{t|t}=\Delta X- H_b(M^0\Delta Y)\nonumber
  \end{equation}
  where we assume $\Delta X$, $\ve$, and $v$ normal.
  Then for the conditional law of $\widetilde{\Delta X}$ given $\Delta Y$ is %we have
$%  \begin{equation}
  %\Lw(\widetilde {\Delta X}|\Delta Y)=
  {\cal N}_p(g, (\EM_p-M^0Z)\Sigma) %\nonumber
$ %  \end{equation}
  for $\Sigma =\Cov\Delta X$ and
$%  \begin{equation}
  g:= M^0\Delta Y - H_b(M^0\Delta Y) = \big(\big|M^0\Delta Y\big|-b\big)_{\SSs +} %\nonumber
$. %  \end{equation}
Hence
  \begin{equation}
  \Lw(\widetilde {\Delta X}) = \Lw(g) \ast {\cal N}_p(0, (\EM_p-M^0Z)\Sigma)\nonumber
  \end{equation}
  which by Cram\'er-L\'evy cannot be normal, as $g$ is obviously not normal.
  Consequently
$%  \begin{equation}
  \Delta X_{t+1} = F_{t+1} \widetilde{ \Delta X_{t}}  + v_{t+1} %\nonumber
$ %  \end{equation}
  cannot be normal either. Hence starting with normal $\Delta X_t$ and $\ve_t$, $\Delta X_{t+1}$ cannot be normal.
  The same assertion clearly holds if $v_t$ is not normal.
  As by \eqref{btbound}, $g_t$ does neither converge to $0$ nor to $M^0\Delta Y$,
  the asymptotic version of Cram\'er-L\'evy also excludes asymptotic normality.
\hfill\qed \smallskip\\

\begin{Rem}\rm\small
 A similar assertion for the case that $v_t$ is normal but not both $\Delta X_t$ and $\ve_t$ are,
  seems plausible and we conjecture that this is true; it may also be proven in particular cases,
  but in general, it is hard to obtain due to the lack of independence of $\Delta X - g$ and $\Delta Y$.
\end{Rem}

\paragraph{Proof to Proposition~\ref{nolin}}
For the second equivalence in Proposition~\ref{nolin} we use the following lemma and a corollary of it:
\begin{Lem}\label{lem1}
  Let $\ve \sim {\cal N}_q(0,V)$, $X\sim P^X$ and for some measurable function $h\colon \mathop{\rm range}(X)\to \R^q$ let $Y=h(x)+\ve$.
  Let $g \in L^l_1(P^X)$, i.e., $g\colon \mathop{\rm range}(X)\to \R^l$ measurable and $\Ew_{P^X} |g(X)| <\infty$. Then
  \begin{equation} \label{lemass}
  \frac{\partial}{\partial y} \Ew [g(X) | Y=y] = \Cov[g(x),  h(x) | Y=y] V^{-1}
  \end{equation}
\end{Lem}
\begin{proof}
For simplicity, we only consider $\rk V=q$; otherwise we may pass to $\ve = A\tilde \ve$ for some $\tilde \ve\sim {\cal N}_{\tilde q}(0,\tilde V)$
with $\rk \tilde V=\tilde q$ and use the generalized inverse $V^-$ instead of $V^{-1}$ everywhere in the proof.

Let $p^\ve$ be the Lebesgue density of $\ve$ and denote $\Lambda^{\ve}(\ve):=\frac{\partial }{\partial  \ve} \log p^\ve(\ve)$.
Then, no matter whether $\ve$ is Gaussian, it holds that
$$
\Ew[g(X)|Y=y]= \frac{\int g(x) p^\ve(y-h(x))\,P^X(dx)}{\int p^\ve(y-h(x))\,P^X(dx)}
$$
As $\ve$ is normal, we may interchange differentiation and integration and obtain that
$$
\frac{\partial}{\partial y} \Ew[g(X)|Y=y]= \Cov[g(X),\Lambda^\ve(Y-h(X))\,|Y=y]
$$
But as $\ve\sim {\cal N}_q(0,V)$, it holds that $\Lambda^\ve(\ve)=-V^{-1} \ve$,
which entails \eqref{lemass}  as
$$\Lambda^\ve(y-h(X))-\Ew[\Lambda^\ve(Y-h(X))| Y=y]=V^{-1}(h(X)-\Ew[h(X)|Y=y])$$
\end{proof}
\begin{Cor}\label{cor1}
In our linear time discrete, Euclidean SSM, ommiting indices $t$, assume that $\rk V=q$ and
let
\begin{equation} \label{Unota}
U:=V^{-1}Z \Delta X,\quad U^0:=U-\Ew[U|\Delta Y], \quad \Delta X^0:=\Delta X-\Ew[\Delta  X|\Delta Y]
\end{equation}
Then
\begin{eqnarray}
  \frac{\partial}{\partial y} \Ew[\Delta X|\Delta Y=y] &= &\Cov(\Delta X, U|\Delta Y=y) \label{cor.asa}\\
  \frac{\partial^2}{\partial y_j\partial y_k} \Ew[\Delta X_i|\Delta Y=y] &=& \Ew(\Delta X_i^0 U_j^0 U_k^0|\Delta Y=y) \label{cor.asb}
\end{eqnarray}
\end{Cor}
\begin{proof}
During the proof we will omit $\Delta$ in notation. Equation~\eqref{cor.asa} is just plugging in Lemma~\ref{lem1}.
We note that equivalently to \eqref{lemass} we could have written
$$
\frac{\partial}{\partial y} \Ew[X|Y=y] =\Ew[X (U^0)^\tau |Y=y]=\Ew[X U^\tau |Y=y]-\Ew[X|Y=y]\Ew[U|Y=y]^\tau
$$
Hence applying Lemma~\ref{lem1} for $g(X)=X_iU_j$ and $g(X)=U_j$ to the last two terms we obtain
\begin{eqnarray*}
 \frac{\partial^2}{\partial y_j\partial y_k} \Ew[X_i|Y=y]&=&\Ew[X_iU_jU_k^0|Y=y]-\Ew [X_i|Y=y] \Ew[U_jU_k^0|Y=y] =\\
  &=&\Ew[X_i^0U_jU_k^0|Y=y]
=\Ew[X_i^0U_j^0U_k^0|Y=y]
\end{eqnarray*}
\end{proof}

\textit{Proof to Proposition~\ref{nolin}}
Equivalence~\eqref{linnorm}:\smallskip\\
If ${\cal L}(\Delta  X)$ is normal, the uncorrelated random variables $\Pi \Delta X$ and $\bar \Pi\Delta X$ are independent and again normal, while
the random variables $\Delta X, \Delta Y$ are jointly
normal, hence linearity of conditional expectation is a well-known fact.

If $\Ew_{\rm\SSs id}[\Delta  X|\Delta Y]$ is linear, after subtracting $\Ew MZ\Delta X$ from both sides,
the defining equation for the conditional expectation $P^Y(dy)$-a.e.\ reads
\begin{equation}\label{1.34}
M \int (y-Zx) p^\ve(y-Zx)\,P^X(dx)=(\EM_p-MZ) \int x p^\ve(y-Zx)\,P^X(dx)
\end{equation}
Let us introduce $q^\ve(y)=yp^\ve(y)$ and the signed measure $Q^X(dx)=x\,P(dx)$; if we
denote the mapping $h\colon \R^q \to \R, y\mapsto h(y)=\int f(y-Zx)\,G(dx)$ by $f \mathop{\ast_{\SSs Z}} G$,
\eqref{1.34} becomes
\begin{equation}\label{1.35}
M q^\ve  \mathop{\ast_{\SSs Z}} P^X =(\EM_p-MZ) p^\ve  \mathop{\ast_{\SSs Z}} Q^X
\end{equation}
We pass over to the Fourier
transforms (denoted with $\,\hat\cdot\,$) for $s\in\R^p$, $t\in\R^q$
\begin{equation}
\begin{array}{ll}
\hat q^X(s) = \int e^{is^\tau x}\, Q^X(dx),&\quad \hat p^X(s)=\int e^{is^\tau x} \,P^X(dx),\\
\hat q^\ve(t) = \int e^{it^\tau x} q^\ve(y) \,dy,&\quad \hat p^\ve(t)=\int e^{it^\tau x} p^\ve(y) \,dy,
\end{array}\nonumber
\end{equation}
As usual, convolution translates into products in Fourier space, in our case
\begin{equation}
\widehat{f \mathop{\ast_{\SSs Z}} G}(t)=\hat f(t) \hat G(Z^\tau t),\qquad t\in\R^q\nonumber
\end{equation}
and hence \eqref{1.35} in Fourier space is
$%\begin{equation}
M \hat q^\ve  \hat p^X(Z^\tau  \,\cdot\,) =(\EM_p-MZ) \hat p^\ve  \hat q^X(Z^\tau \,\cdot\,)%\nonumber
$. %\end{equation}
For the derivatives $(\hat p^X)'(s)$, $(\hat p^\ve)'(t)$ for $s\in\R^p$ and $t\in\R^q$, we obtain
\begin{eqnarray}
(\hat p^X)'(s)=i\,\hat q^X(s), \qquad
(\hat p^\ve)'(t)=i\,\hat q^\ve(t)
\label{diffF2}
\end{eqnarray}

By assumption, $\EM_p-MZ$ is invertible and $\ve\sim{\cal N}_q(0,V)$, hence
 $\hat p^\ve(t)=\exp(-t^\tau V t/2)>0$ and together with \eqref{diffF2}, this gives the
linear differential equation
\begin{equation} \label{diffeq}
(\hat p^X)'(Z^\tau t)= -(\EM_p-MZ)^{-1} M V t    \hat p^X(Z^\tau  t)
\end{equation}
Fixing any direction $t_0$ such that $Z^\tau t_0\not=0$, this becomes an ODE
\begin{equation}
g'(s)= -t_0^\tau Z(\EM_p-MZ)^{-1} M V t_0 s   g(s),\qquad g(0)=1 \nonumber
\end{equation}
which has a unique solution given by
\begin{equation}
g(s)= \exp(-t_0^\tau Z(\EM_p-MZ)^{-1} M V t_0 s^2/2) \nonumber
\end{equation}
This is the characteristic function of a normal distribution, so $Z\Delta X$,
hence also $\bar\Pi\Delta X$ are normal, and  together with \eqref{PiNorm}
the assertion follows. On the other hand, $\Cov Z\Delta X=Z\Sigma Z^\tau$,
so we have also shown that $Z(\EM_p-MZ)^{-1} M V=Z\Sigma Z^\tau$, which otherwise is tricky unless assuming
 $\Sigma$ and $\Delta$ invertible.
\smallskip\\
Equivalence~\eqref{3mom}:\smallskip\\
If $\Ew_{\rm\SSs id}[\Delta  X|\Delta Y]$ is linear, by equivalence~\eqref{linnorm} $\Delta X$ and $\Delta Y$ are jointly normal
with expectation $0$, so the conditional law of $\Delta X$ given $\Delta Y$ is again normal with expectation 0, hence
in particular symmetric so the assertion follows.\\
Now assume
\begin{equation}\label{3rdmom0}
\Ew\Big[\Big(e^\tau (\Delta X - \Ew[\Delta X|\Delta Y])\Big)^3\,\Big|\,\Delta Y\Big]=0\qquad  \forall \, e\in\R^p
\end{equation}
Apparently,
$\Ew_{\rm\SSs id}[\Delta  X|\Delta Y]$ is linear iff
$\partial^2 / \partial y \partial y^\tau \Ew_{\rm\SSs id}[\Delta  X|\Delta Y]=0.$
But Corollary~\ref{cor1} gives (in the notation of \eqref{Unota})
\begin{equation} \label{XiUjUk}
\frac{\partial^2}{\partial y_j\partial y_k} \Ew[\Delta X_i|\Delta Y=y] = \Ew(\Delta X_i^0 U_j^0 U_k^0|\Delta Y=y)
\end{equation}
By complete polarization (compare \citet[Chap.~I.1]{Weyl:97}), \eqref{3rdmom0} also entails that the symmetric
multilinear form given by $\Ew [\Delta X_i^0\Delta X_j^0\Delta X_k^0|Y=y]_{i,j,k\in \{1,\ldots,p\}}$ is identically
$0$. So the assertion follows, as with $\tilde Z=ZV^{-1}$, the RHS of \eqref{XiUjUk} is just
$$\Ds\sum\nolimits_{h,l=1}^p \tilde Z_{j,h}\tilde Z_{k,l}\Ew(\Delta X_i^0 \Delta X_h^0 \Delta X_l^0|\Delta Y=y)\vspace{-5ex}$$
\hfill\qed

\paragraph{Proof to Theorem~\ref{ThmeSO}}
We proceed as in Theorem~\ref{ThmSO}, but note that in the {\rm eSO} context \eqref{SOdistr} becomes
\begin{eqnarray}
P(X \in A, Y^{\rm \SSs re} \in B) &=& (1-r)  \int \Jc_A(x)\Jc_B(y) \pi(y,x)\,P^{X^{\rm\SSs id}}(dx)\,\mu(dy)\nonumber\\
 &&
\label{eSOdistr}\quad
 +r \int \Jc_A(x)\Jc_B(y) q(y)\, P^{X^{\rm\SSs di}}(dx)\,\mu(dy)\nonumber
\end{eqnarray}
and hence \eqref{SOEw} becomes
\begin{equation} \label{eSOEw}
\Ew_{\rm\SSs re}[X|Y^{\rm \SSs re}=y]=\frac{r q(y)\Ew_{\rm\SSs di}[X^{\rm\SSs di}]+(1-r)p^{Y^{\rm\SSs id}}(y)\Ew_{\SSs \rm id}[X|Y]}{%
r q(y)+(1-r)p^{Y^{\rm\SSs id}}(y)} \nonumber
\end{equation}
But by \eqref{esogl}, the RHS of \eqref{eSOEw} is exactly $F(q)$ from \eqref{SOEw}. Thus, we may jump to the proof of
Theorem~\ref{ThmSO} from this point on, replacing $\tr \Cov X$ by
\begin{equation}
\tilde G := \tr \Cov_{P_0^{X^{\rm \SSs di}}} X^{\rm \SSs di}= G - |\Ew_{\rm \SSs id} X^{\rm \SSs id}|^2 \nonumber
\end{equation}
in equation~\eqref{9.7}. For passing from $\partial {\cal U}^{\rm\SSs eSO}$ to ${\cal U}^{\rm\SSs eSO}$,
 let $f_r$, $\hat P_r\otimes \hat Q_r$ be the
 components of the saddle-point at
 $\partial {\cal U}^{\rm\SSs eSO}(r)$ and $R(f,P\otimes Q,r)$ be the MSE of procedure $f$ at
 $\partial {\cal U}^{\rm\SSs eSO}(r)$ with contaminating $P^{Y^{\rm \SSs di}}\otimes P^{X^{\rm \SSs di}}=P\otimes Q$.
Instead of equation~\eqref{varsplit}, we use
\begin{equation}
\Delta G := \tilde G-\tr \Cov_{\rm\SSs id} X^{\rm\SSs id}=G-\Ew_{\rm \SSs id} |X^{\rm \SSs id}|^2 \ge 0 \nonumber
\end{equation}
and abbreviating $R(f,P\otimes Q,r)-R(f,P\otimes Q,0)$
by $\tilde R(f,P\otimes Q,r)$ we obtain
\begin{eqnarray}
&&\hspace{-3em}\tilde R(f_s,P \otimes Q,r )=
r \,\Big\{\tr \Cov_Q X^{\rm\SSs di} - \Cov_{\rm\SSs id} X^{\rm\SSs id} + \Ew_{P}[\min(|D(Y^{\rm\SSs di})|,\rho(s))^2]
\,\Big\}\leq \nonumber\\
&\leq& r \Big\{\Delta G+\Ew_{\rm\SSs id}\Big[|D(Y^{\rm\SSs id})|^2 \bar w_s(Y^{\rm\SSs id})\Big] + \rho(s)^2\,\Big\} \nonumber = \\
&= &\!\!\tilde R(f_s,\hat P_r\otimes \hat Q_r,r )<\tilde R(f_s,\hat P_r\otimes \hat Q_r,s )=
\tilde R(f_s,\hat P_s\otimes\hat Q_s,s ) \nonumber
\end{eqnarray}
Hence the saddle-point extends to ${\cal U}^{\rm\SSs eSO}(r)$.
\eqref{sadvaleSO} follows by plugging in the results.
\hfill\qed

\paragraph{Proof to Proposition~\ref{testprop}}
Under $H_0$, due to Proposition~\ref{nolin}, $\Delta X_i^\natural\iid{\cal N}_p(0,\Sigma)$.
Hence $e^\tau\Delta X_i^\natural\iid{\cal N}(0,\sigma^2)$. Thus by the Lindeberg-L\'evy CLT,
$$
\frac{1}{\sqrt{n}} \sum_{i=1}^n (e ^\tau \Delta X_i^\natural)^3\wto {\cal N}(0, \Ew [(e ^\tau \Delta X_t)^6])
$$
But the sixth moment of ${\cal N}(0,\sigma^2)$ is just $15 \sigma^6$. Hence by the assumed consistency of $\hat e_n$ for $e$,
Slutsky's Lemma yields \eqref{asynorm}. Asymptotically, the testing problem is a test for a normal mean $\mu$ to be $0$ or not,
which yields the corresponding optimality for the Gauss test given in \eqref{Gausstest}.
\hfill\qed

\paragraph{Proof to Proposition~\ref{esoprop}}
 Let us identify $X\leadsto \Delta  X^{\cal N}$,
 $Y\leadsto \Delta  Y^{\cal N}:=Z\Delta  X^{\cal N}+\ve$, and set $P^\ve={\cal N}_q(0,V)$,
 $P^X={\cal N}_p(0,\Sigma)$, and let $p^\ve$ the corresponding Lebesgue density, then
 $\pi(y,x)=p^\ve(y-Zx)$.
 Assertions (1') and (3') of Theorem~\ref{ThmeSO} show that the {\rm eSO}-optimal %procedure
 $f_0$ in our ``Bayesian''
 model of subsection~\ref{1stpOpt} is just
 $
f_0(y)=M^0(y)\min\{1, \rho/\big|M^0 y\big|\}
$
with $\rho$ according to \eqref{P0def} such that $\int \,dP_0^{Y^{\rm\SSs di}}=1$
and $M^0=\Sigma Z^\tau(Z\Sigma Z^\tau+V)^{-1}$.\\ By assumption, $\Delta X^{\rm\SSs rLS}$ lies
in the corresponding {\rm eSO}-neighborhood ${\cal U}(r)$ about $\Delta X^{\cal N}$ so
the value of the saddle-point from equation~\eqref{sadvalSO} is also a bound for the MSE of $X^{\rm\SSs rLS}_{t|t}$
on ${\cal U}(r)$.
\hfill\qed
\begin{Rem}\rm\small
One should mention, however, that due to assumption~\eqref{indep} resp.\ \eqref{indep2}, members of an SO-neighborhood
${\cal U}'(r')$ about $\Lw(\Delta X^{\rm\SSs rLS}, \Delta Y^{\rm\SSs rLS})$ need not lie in an {\rm eSO} neighborhood ${\cal U}(r+r')$
about $\Lw(\Delta X^{\cal N}, \Delta Y^{\cal N})$.
\end{Rem}

%-------------------------------------------------------------------------------
\section*{Acknowledgements}
%-------------------------------------------------------------------------------
The author would like to acknowledge and thank for the stimulating discussion he had with
Gerald Kroisandt at ITWM which led to the definition of rLS.IO. %
He also wants to thank Helmut Rieder for several suggestions as to notation and
formulations which have much improved clarity and readability of this paper.
Many thanks go to Nataliya Horbenko for proof-reading this paper.
Of course, the opinions expressed in this paper as well as any errors are
solely the responsibility of the author.
%-------------------------------------------------------------------------------

%-------------------------------------------------------------------------------
\end{document}